\documentclass{article}
\usepackage{spconf}
\usepackage{cite}
\usepackage{amsmath,amssymb,amsfonts,dsfont,amsthm}
\usepackage{algorithmic}
\usepackage{graphicx}
\usepackage{textcomp}
\usepackage{xcolor}
\usepackage{subfig}

\newcommand{\E}{\mathbb{E}}
\newcommand{\prob}{\mathbb{P}}
\newcommand{\p}{\mathbb{P}}
\DeclareMathOperator{\reals}{\mathbb{R}} 
\DeclareMathOperator{\integers}{\mathbb{Z}} 

\def\cE{{\mathcal E}}
\def\cF{{\mathcal F}}

\def\cN{{\mathcal N}}

\def\cT{{\mathcal T}}

\newtheorem{theorem}{Theorem}

\newtheorem{lemma}{Lemma}

\newtheorem{corollary}{Corollary}

\def \QED {\hfill{$\Box$}}
\newenvironment{proofof}[1]{\noindent {\em Proof of #1.  }}{\QED}

\DeclareMathOperator*{\argmin}{arg\,min}
\newcommand{\norm}[1]{\left \| #1 \right \|}
\DeclareMathOperator*{\esssup}{ess\,sup}

\newcommand{\cut}[1]{}
    
\begin{document}

\title{Bayes-optimal Methods for Finding the Source of a Cascade}

\name{Anirudh Sridhar, H. Vincent Poor}
\address{Princeton University \\ Department of Electrical Engineering \\ Princeton, NJ}

\maketitle

\begin{abstract}
We study the problem of estimating the source of a network cascade. The cascade initially starts from a single vertex and spreads deterministically over time, but only a noisy version of the propagation is observable. The goal is then to design a stopping time and estimator that will estimate the source well while ensuring the number of affected vertices is not too large. We rigorously formulate a Bayesian approach to the problem. If vertices can be labelled by vectors in Euclidean space (which is natural in spatial networks), the optimal estimator is the conditional mean estimator, and we derive an explicit form for the optimal stopping time under minimal assumptions on the cascade dynamics. We study the performance of the optimal stopping time on lattices, and show that a computationally efficient but suboptimal stopping time which compares the posterior variance to a threshold has near-optimal performance.  
\end{abstract}


\section{Introduction}

 Network dynamics are often unstable: the behaviors of a small subset of vertices may rapidly disseminate to the rest of the network. This type of instability, known as a network cascade, has been observed in diverse applications such as the spread of diseases in geographical networks \cite{CF2010,flubreaks,Ant2015}, malware in a computer networks \cite{KW91,MI17}, and fake news in social networks \cite{tacchini2017some,AG17,Fourney2017}. When such cascading failures are present in a network, it is of utmost importance to find the source as fast as possible. Unfortunately, in many cases of interest the cascade is not directly observable. For instance, if an epidemic spreads over a contact network and an individual falls sick, it could be a symptom of the disease or it could be due to exogenous factors (e.g. allergies). Over time, one may better distinguish between these possibilities and construct better source estimates at the cost of allowing the cascade to propagate even further. An optimal algorithm in our framework will achieve the best possible tradeoff between estimation error and the number of agents affected by the cascade. 
 
 In this paper we provide a Bayesian approach to optimal source estimation. If vertices can be labelled by vectors in Euclidean space, we derive an explicit form for the optimal source estimation algorithm for a simple, but illustrative cascade model. We then study its performance of our source estimator in lattices, pinning down the runtime of the optimal algorithm for a certain class of Bayes priors for the source vertex. Though the optimal stopping time has a complex form, we show that a simple algorithm which compares the posterior variance to a threshold enjoys orderwise optimal performance in lattices. 
\subsection{A model of network cascades with noisy observations}
\label{subsec:model}

Let $G$ be a graph with vertex set $V$ and let time be indexed by a positive integer $t$. Initially, at $t = 0$, a single vertex $v_0$ is affected by the cascade. The cascade then spreads deterministically from $v_0$ over time, so that a vertex $u$ is affected by the cascade when $d(u,v_0) \le t$, where $d(\cdot, \cdot)$ denotes the shortest-path distance in $G$. Although the cascade is not directly observable, we assume a system monitor has access to {\it noisy} signals from each vertex, where the signal corresponding to $u$ at time $t$ is given by $y_u(t)$. Conditioned on the cascade source, the signals are independent over time, with 
$$
y_u(t) \sim \begin{cases}
Q_0 & d(u,v_0) > t \\
Q_1 & d(u,v_0) \le t.
\end{cases}
$$
where $Q_0$ and $Q_1$ are two mutually absolutely continuous probability measures. We can think of $y_u(t) \sim Q_0$ being typical behavior and $y_u(t) \sim Q_1$ as anomalous behavior caused by the cascade.\footnote{This models a variety of data-gathering methods, including noisy measurements and random sampling.} This type of model has been studied in recent literature in the context of cascade source estimation \cite{sridhar_poor_2019} and quickest detection of cascades \cite{ZouVeeravalli2018,ZVLT2018,ZVLT2019_IEEE,Rovatsos2019,qcd_empirical}.
\cut{
We make a few remarks about the distributions $Q_0$ and $Q_1$. We can think of $y_u(t) \sim Q_0$ being typical behavior and $y_u(t) \sim Q_1$ as anomalous behavior caused by the cascade. For instance, if we consider an epidemic spreading on a contact network, $Q_0$ and $Q_1$ could be the distribution of an individual's body temperature, in which case we may expect the mean of $Q_1$ to be larger than the mean of $Q_0$. 

Another important example which falls under the purview of this model is {\it random testing}. Suppose that each individual in the network gets tested at a particular time with probability $p$. The test is not perfect, and outputs an incorrect diagnosis with probability $\epsilon$ and a correct one with probability $1 - \epsilon$. Formally, for any $u \in V$, we have $y_u(t) \in \{0,1,\times\}$ where $y_u(t) = 0$ (resp., $y_u(t) = 1$) if $u$ is tested and the test indicates $u$ is not infected (resp., infected), and $y_u(t) = \times$ if $u$ is not tested at time $t$. In this situation, we have 
\begin{align*}
Q_0 & : = (q_{0}, q_{1}, q_{\times} ) = (p(1 - \epsilon), p \epsilon, 1 - p); \\
Q_1 & : = (q_{0}, q_{1}, q_{\times} ) = (p \epsilon, p(1 - \epsilon), 1 - p).
\end{align*}
Above, $q_{0}$ is the probability of sampling 0, with similar interpretations for $q_1, q_{\times}$. 
}

\subsection{Formulation as a stochastic optimization problem}

Let $(\Omega, \mathbb{P}, \cF)$ be a common probability space for all random objects, and let $\{\cF_t \}_{t = 0}^\infty$ be the natural filtration formed by the public states: $\cF_t := \sigma(y(0), \ldots, y(t) )$ where $y(t) : = \{y_u(t) \}_{u \in V}$. For a vertex $v$ and $t \ge 0$, we define $\cN_v(t)$ to be the set of vertices within distance $t$ of $v$, and $\partial \cN_v(t)$ is the set of vertices that are {\it exactly} distance $t$ from $v$. Any algorithm for estimating the cascade source $v_0$ may be represented by $(T,\hat{v})$, where $T$ is a stopping time and $\hat{v} = \left \{ \hat{v}(t) \right \}_{t = 0}^\infty$ is a sequence of source estimators, $\hat{v}(t)$ being $\cF_t$-measurable. The problem of finding an estimation algorithm that achieves the best tradeoff between the accuracy of an estimator, measured by $d(\hat{v}(t),v_0)$, and the number of affected vertices, given by $|\cN_{v_0}(t)|$ is captured by the following stochastic optimization problem: 
\begin{equation}
\label{eq:optimization_problem}
\min\limits_{T, \hat{v}} \mathbb{E} \left[ d(v_0,\hat{v}(T)) +  |\cN_{v_0}(T)| \right],
\end{equation}
where we assume, for simplicity, that $v_0$ is selected uniformly at random from the set of vertices. The formulation in \eqref{eq:optimization_problem} implies that the optimal estimator $\hat{v}(t)$ conditioned on $\cF_t$ minimizes the risk defined by the distance function: 
\begin{equation}
\label{eq:optimal_estimator}
\hat{v}(t) = \argmin_{v \in V} \E [ d(v_0,v) \mid \cF_t ].
\end{equation}
If $d$ is the shortest-path distance, the estimator is a complex function of the graph topology and the past observations. However, if vertices can be labelled by vectors in $k$-dimensional Euclidean space -- a natural assumption in spatial or geographical networks where vertices represent locations -- the analysis simplifies considerably. Given a probability distribution over the vertices of the graph, we can now compute basic statistical quantities such as the expected value and variance of a random variable sampled from the distribution. If we replace $d(v_0, \hat{v}(T))$ in \eqref{eq:optimization_problem} with $\norm{v_0 - \hat{v}(T)}_2^2$, the optimal estimator at a given time is exactly the conditional mean estimator: $\hat{v}(t) = \E [ v_0 \mid \cF_t]$. The conditional mean estimator enjoys a variety of mathematical properties such as a martingale structure and consistency\footnote{In \cite{sridhar_poor_2019} it was shown that $\frac{\p(v_0 = v\mid \cF_t)}{ \p(v_0 = u \mid \cF_t)} \to \infty$ as $t \to \infty$ if $v$ is the true source and $u$ is any other vertex, which implies consistency if the vertex set is finite.} which we will heavily exploit in deriving the optimal stopping time.  

\subsection{Related work} 

Shah and Zaman first studied the problem of estimating the source of a network cascade \cite{shah2011rumors,shah2010detecting}. They assume that after the cascade has spread for a long time, a snapshot of the set of affected vertices is perfectly observed. The goal is then to estimate the source based on this single snapshot. Several authors have built on this work, deriving confidence intervals for the source as well as considering various cascade models \cite{KhimLoh15,ying_zhu_source_localization}. This observation model is however not well-suited for real-time settings, in which one often has access to {\it streaming} data that is biased or noisy. The observation model in Section \ref{subsec:model}, on the other hand, naturally captures this latter setting. 

Our work naturally falls under the growing body of work on {\it sequential} inference of cascades, which assumes access to noisy streaming data (as opposed to inference from a noiseless snapshot) generated by the variants of the model in Section \ref{subsec:model}. Most of this literature has studied the quickest detection problem, which aims to detect with minimum delay when the cascade affects a certain number of vertices \cite{ZouVeeravalli2018,ZVLT2018,ZVLT2019_IEEE,Rovatsos2019,qcd_empirical}. The closest work to ours is by Sridhar and Poor \cite{sridhar_poor_2019}, who study the source estimation problem in a non-Bayesian setting. By phrasing source estimation as a multi-hypothesis testing problem, they demonstrate how to design matrix sequential probability ratio tests that are asymptotically optimal in terms of minimizing the expected runtime as the number of nodes in the network tend to infinity and the Type I error tends to zero. Interestingly, though our approach yields a drastically different algorithm for source estimation, the runtime of our algorithm matches that of \cite{sridhar_poor_2019} in certain regimes and outperforms it in others. We further remark that our results rely on the somewhat restrictive assumption of having vector-labeled vertices, while the results of \cite{sridhar_poor_2019} do not make such an assumption. Generalizing our results to larger graph families is an important subject of future work. 

\cut{
There are a number of recent works on the related problem of detecting network cascades. Zou, Veeravalli, Li, Towsley and Rovatsos study the problem of quickest detection of the event where a specified number of vertices are affected by the cascade \cite{ZouVeeravalli2018,ZVLT2018,ZVLT2019_IEEE,Rovatsos2019}. Their formulation is non-Bayesian and assumes that the cascade dynamics are unknown. As the cascade growth rate becomes slower, they show that tests based on cumulative log-likelihood ratios are asymptotically optimal. In further work, Zhang, Yao, Xie and Qiu designed algorithms for quickest detection when the cascade dynamics follows the Independent Cascade model  \cite{independent_cascade}, and studied their performance empirically \cite{qcd_empirical}. 
}

\section{Deriving the Bayes-optimal solution}
\label{sec:bayes}

We begin by making a few assumptions to simplify our analysis. The underlying graph $G = (V,E)$ is assumed to be connected, infinite, and locally finite\footnote{A graph is locally finite if every vertex has finite degree.}, which is a common assumption in the source estimation literature that allows us to ignore boundary effects (e.g., when the cascade reaches all the vertices). We further assume that $G$ is vertex-transitive\footnote{A graph is vertex-transitive if, for every pair of vertices $u,v \in V$, there is an automorphism $A$ such that $A(u) = v$.} which implies, in particular, that 
\begin{equation}
\label{eq:neighborhood_size}
| \cN_u(t) | = | \cN_v(t) |, \qquad \forall u,v \in V, t \ge 0.
\end{equation}
In light of \eqref{eq:neighborhood_size}, we will often write $|\cN(t)|$ instead of $|\cN_u(t)|$ as the size of the set does not depend on the choice of vertex $u$; we similarly write $| \partial \cN_u(t) | = | \partial \cN(t)|$. Examples of graphs which satisfy the outlined assumptions are regular trees and lattices. We expect that these assumptions may be relaxed to capture more realistic networks, and this will be a subject of future work. 

The following result describes the optimal solution to \eqref{eq:optimization_problem}.

\begin{theorem}
\label{thm:cascade_optimal_stopping}
Suppose that the prior distribution for $v_0$ has finite variance. Let $\cT_s$ be the set of stopping times such that $T \ge s$ a.s. for $T \in \cT_s$. Define, for any stopping time $T \in \cT_s$, the random variable 
\begin{equation}
\label{eq:f_s}
f_s(T) : = \norm{\hat{v}(T) - \hat{v}(s)}_2^2 - \left( |\cN(T)| - |\cN(s)| \right).
\end{equation}
Then the optimal stopping time that solves \eqref{eq:optimization_problem} is 
$$
T_{opt} : = \min \left \{ s \ge 0 : \sup\limits_{T \in \cT_s} \E \left[   f_s(T)  \mid \cF_s \right] = 0 \right \}.
$$
\end{theorem}

We can interpret the optimal stopping time as follows. The quantity $\E\left[ \norm{\hat{v}(T) - \hat{v}(s)}_2^2 \mid \cF_s \right]$ is the expected amount of information gained about $v_0$ at time $T$, conditioned on current information. On the other hand, $| \cN(T)| - |\cN(s)|$ is the number of new infections during this time. If, for every $T$, the information gained is less than the number of added infections, then it is not worth it to take even a single extra step. Conversely, if there is some $T \in \cT_s$ where the information gained is greater than the number of new infections until that point, then it is worth it to keep sampling. 

The proof of the theorem relies on a result from optimal stopping theory, which we briefly review. Let $\{Y_t \}_t$ be an adapted sequence of stochastic rewards, so that $Y_t$ is $\cF_t$-measurable. Let $\cT$ be the set of stopping times. The goal is to find a stopping time $T \in \cT$ that achieves $\sup_{T \in \cT} \E[ Y_T ]$. For any integer $k \ge 0$, define $\gamma_k : = \esssup_{T \in \cT_k} \E [ Y_T \mid \cF_k ]$. Informally, $\gamma_k$ is the maximum expected reward possible, given the information at time $k$. The following result gives a closed-form expression for the optimal stopping time in terms of $\gamma_k$. 

\begin{theorem}[Theorem 3.7 in \cite{poor_hadjiliadis_2008}]
\label{thm:optimal_stopping}
If $\E \left[ \sup_k Y_k^+ \right] < \infty$, then the stopping time $T_{opt} : = \min \{ k \ge 0 : Y_k = \gamma_k \}$ is optimal, in the sense that it achieves $\sup_{T \in \cT} \E[ Y_T]$. 
\end{theorem}

We will now use this result to prove Theorem \ref{thm:cascade_optimal_stopping}. 

\begin{proofof}{Theorem \ref{thm:cascade_optimal_stopping}}
We begin by reformulating \eqref{eq:optimization_problem} to be in the optimal stopping framework. By orthogonality of martingale increments and consistency of the conditional mean estimator, we can decompose $\E [ \| \hat{v}(k) - v_0 \|_2^2 ]$ as 
$$
\E \left[ \norm{\hat{v}(0) - v_0}_2^2 \right] - \E \left [ \norm{\hat{v}(k) - \hat{v}(0)}_2^2 \right ].
$$
It follows that optimal stopping time for \eqref{eq:optimization_problem} also achieves
$$
\sup\limits_{T \in \cT} \E \left[\norm{\hat{v}(T) - \hat{v}(0)}_2^2 - |\cN(T)| \right].
$$
To apply Theorem \ref{thm:optimal_stopping} it suffices to check that 
$$
\E \left[ \sup_k  \norm{\hat{v}(k) - \hat{v}(t)}_2^2 \right] < \infty,
$$
which follows from finiteness of $\E  [ \| v_0 - \hat{v}(0)\|_2^2 ]$ and orthogonality of martingale increments. Applying Theorem \ref{thm:optimal_stopping}, we see that 
\begin{multline*}
T_{opt} : = \min \left \{ s \ge 0:  \norm{\hat{v}(s) - \hat{v}(0)}_2^2 - |\cN(s)| \right. \\
\left. = \sup\limits_{T \in \cT_s} \E \left[  \norm{\hat{v}(T) - \hat{v}(0)}_2^2 - |\cN(T)| \mid \cF_s \right] \right \}.
\end{multline*}
Rearranging and again invoking the orthogonality of martingale increments, the condition in the stopping time becomes 
\begin{equation*}
\sup\limits_{T \in \cT_s} \left( \E \left[  \norm{\hat{v}(T) - \hat{v}(s) }_2^2  \mid \cF_s \right] - |\cN(T)| + |\cN(s)| \right)= 0.
\end{equation*}
\end{proofof}

We present two simpler stopping times in the following corollary; the proof follows easily from the structure of $T_{opt}$. 

\begin{corollary}
\label{cor:optimal_stopping}
Let $r$ be a positive integer. Recall the definition of $f_s$ from \eqref{eq:f_s} and define the stopping times
\begin{align*}
T_r &:  = \min\left \{ s \ge 0 : \E[f_s(r)\mid \cF_s] \le 0 \right \}. \\
T_+ &: = \min \left \{ s \ge 0 : \E \left[ \norm{v_0 - \hat{v}(s) }_2^2 \mid \cF_s \right] \le |\partial N(s)| \right \}. 
\end{align*}
Then $T_r \le T_{opt} \le T_+$ almost surely. 
\end{corollary}  

We remark that the stopping time $T_+$ may be a desirable alternative to $T_{opt}$ in practice since it is straightforward to compute from the posterior distribution and directly gives a bound on estimation error. Fortunately, as we shall see in the following section, $T_+$ is nearly optimal in lattices. 

\section{Performance analysis in lattices}
\label{sec:performance}

In this section we characterize the performance of the optimal stopping time on $d$-dimensional lattices. We highlight the main ideas and leave the details to an appendix \cite{sridhar2020bayesoptimal} for the interested reader. First, we introduce some notation and assumptions. Let $G = (V,E)$ be the $d$-dimensional lattice with vertex set $V$ is given by integer coordinates of $\reals^d$ so that $d(u,v) = \norm{u - v}_1$,\footnote{$\norm{x}_1 : = \sum_{i = 1}^d |x_i|$ denotes the $\ell_1$ norm.} and two vertices $u,v$ are adjacent if and only if $d(u,v) = 1$. For each $v \in V$, define the measure $\mathbb{P}_v : = \mathbb{P}(\cdot \mid v_0 = v)$. We next define the {\it neighborhood growth function} $h(t) = \sum_{s = 0}^t | \cN(t)|$ as well as its inverse $H = h^{-1}$. At a high level, $h(t)$ captures the rate of information gain and thus plays a central role in our analysis. In the $d$-dimensional lattice, $h(t) \asymp t^{d + 1}$ and $H(z) \asymp z^{\frac{1}{d + 1}}$ \cite[Lemma 3]{sridhar2020bayesoptimal}.

An important goal is to characterize the performance of our source estimator as a function of the number of vertices in the graph. We will therefore assume that the prior distribution of $v_0$ is uniform over a finite vertex set $V_n \subset V$, where the number of vertices in $V_n$ is on the order of $n$. We assume without loss of generality that $V_n = \cN_0(k)$ for some $k$.\footnote{This ensures that all vertices are as close as possible to each other, which only makes the problem harder. The center of the neighborhood can be chosen arbitrarily as lattices are vertex-transitive.} Choosing $k = n^{1/d}$ (assuming that $n^{1/d}$ is an integer) ensures that $|V_n| \asymp n$ \cite[Lemma 3]{sridhar2020bayesoptimal}. We let $T_r^n, T_{opt}^n$ and $T_+^n$ denote the stopping times corresponding to the initial uniform prior for $v_0$ over $V_n$. The following theorem characterizes the performance of both $T_{opt}^n$ and $T_+^n$ as $n \to \infty$.

\begin{theorem}
\label{thm:lattice_performance}
There exist constants $\Delta_1 \le \Delta_2$ depending only on $Q_0,Q_1$ and the lattice dimension $d$ such that 
$$
\lim\limits_{n \to \infty} \min\limits_{v \in V_n} \p_v \left( H(\Delta_1 \log n) \le T_{opt}  \le H(\Delta_2 \log n) \right) = 1.
$$
Furthermore, the same holds for $T_+$. 
\end{theorem}

In $d$-dimensional lattices, this implies that $T_{opt}^n \asymp T_+^n \asymp (\log n)^{\frac{1}{d + 1}}$ with high probability. We remark that this matches the behavior of the non-Bayes optimal algorithm in certain regimes, and outperforms it in others \cite{sridhar_poor_2019}.

The proof of Theorem \ref{thm:lattice_performance} follows from characterizing the variance of the posterior distribution. The following lemma shows that initially, the posterior variance remains large for a long time. 

\begin{lemma}
\label{lemma:variance_lower_regime}
There is a constant $\Delta_1 = \Delta_1(Q_0,Q_1,d)$ such that for $n$ sufficiently large, if $t \le H \left(\Delta_1 \log n \right)$, then there are constants $a_1, b_1  > 0$ depending only on the dimension $d$ of the lattice such that 
$$
\min\limits_{v \in V_n }\p_v \left( \E [ \norm{ v_0 - \hat{v}(t)}_2^2 \mid \cF_t ] \ge a_1n^{2/d} \right) \ge 1- b_1n^{-1/2}.
$$
\end{lemma}

In words, Lemma \ref{lemma:variance_lower_regime} shows that in the initial stages of the cascade spread, the level of uncertainty about the source location is quite large - this is essentially because the number of infections is too small to distinguish anomalies from random noise. To see this more rigorously, first observe that $\pi_u(t)$ is the posterior probability that $u$ is the source, then $\pi_u(t)$ primarily depends on $\{y_w(s)\}_{s \le t}$ for $w \in \cN_u(t)$; this is because if $u$ is the source, $\cN_u(t)$ is precisely the set of infected nodes. Hence if a pair of vertices $u, u'$ are far from the source, the signals within their neighborhoods are identically distributed so we expect $\pi_u(t) \approx \pi_{u'}(t)$ - the posterior is roughly uniform over vertices far from the source. If $t$ is not too large so that $\pi_{v_0}(t)$ is not too much larger than $\pi_u(t)$ and $\pi_{u'}(t)$, we expect that the posterior variance will not change significantly from its initial value. Technically, these ideas are carried out by first bounding the covariance between $\pi_u(t)$ and $\pi_{u'}(t)$ for each pair of vertices $u,u'$ and then applying basic concentration inequalities to characterize the posterior variance. We defer the details to \cite{sridhar2020bayesoptimal}.

The next lemma shows that the posterior variance exhibits a {\it sharp transition}: after a certain point, the posterior variance rapidly approaches zero. 

\begin{lemma}
\label{lemma:variance_upper_regime}
There are constants $\Delta_2,a_2,b_2$ depending only on $Q_0$ and $Q_1$ such that for $n$ sufficiently large and $t \ge H( \Delta_2 \log n)$,
$$
\min \limits_{v \in V_n} \p_v \left( \E [ \norm{v_0 - \hat{v}(t) }_2^2 \mid \cF_t] \le e^{-a_2 t} \right) \ge 1 - e^{-b_2t}.
$$
\end{lemma}

\noindent The main idea behind the proof is that when $t \ge H(\Delta_2 \log n)$, $\pi_u(t)$ is relatively large for $u$ close to the source $v_0$ and quite small for $u$ far from the source. The key component is a large-deviations estimate for the likelihood ratio $\pi_u(t)/\pi_{v_0}(t)$ which was derived in \cite{sridhar_poor_2019}.

The proof of Theorem \ref{thm:lattice_performance} follows directly from the Lemmas \ref{lemma:variance_lower_regime} and \ref{lemma:variance_upper_regime}.

\begin{proofof}{Theorem \ref{thm:lattice_performance}}
In light of Corollary \ref{cor:optimal_stopping}, it suffices to prove that $T_r^n \ge H(\Delta_1 \log n)$ and $T_+ \le H(\Delta_2 \log n)$ with high probability. Since $|\partial \cN(t) | \ge 1$ for all $t$, Lemma \ref{lemma:variance_upper_regime} implies that 
$$
\lim\limits_{n \to \infty} \max\limits_{v \in V_n} \p_v \left( T_+ \le H(\Delta_2 \log n) \right) = 1. 
$$
Next, set $r = H(\Delta_2 \log n)$; we will show that $T_r \ge H(\Delta_1 \log n)$ with high probability, which amounts to providing a lower bound for $\E [ \norm{\hat{v}(r) - \hat{v}(t)}_2^2 \mid \cF_t]$, where we take $t \le H(\Delta_1 \log n)$. We can write $\E [ \norm{\hat{v}(r) - \hat{v}(t)}_2^2 \mid \cF_t ]$ as 
\begin{equation}
\label{eq:thm_var_decomp}
\E [ \norm{v_0 - \hat{v}(t)}_2^2 \mid \cF_t ] - \E [ \norm{v_0 - \hat{v}(r)}_2^2 \mid \cF_t ],
\end{equation}
which follows from orthogonality of martingale increments. Lemma \ref{lemma:variance_lower_regime} provides a lower bound for the first term in \eqref{eq:thm_var_decomp}, so we proceed by deriving an upper bound for the second term. Define the $\cF_r$-measurable event $\cE : = \{ \E [ \norm{v_0 - \hat{v}(r)}_2^2 \mid \cF_r ] \le 1\}$, and decompose $\E [ \norm{v_0 - \hat{v}(r)}_2^2 \mid \cF_t ]$ as 
\begin{multline*}
\E [ \norm{v_0 - \hat{v}(r)}_2^2 \mid \cF_t ] = \E [\E [ \norm{v_0 - \hat{v}(r)}_2^2 \mid \cF_r ] \mathds{1}_{\cE} \mid \cF_t] \\
+ \E [ \E [ \norm{v_0 - \hat{v}(r)}_2^2 \mid \cF_r] \mathds{1}_{\cE^c} \mid \cF_t].
\end{multline*}
The first term is bounded by 1 and the second is bounded by $\max_{u,v \in V_n} \norm{u - v}_2^2 \p (\cE^c \mid \cF_t) \le 4n^{2/d} \p(\cE^c \mid \cF_t)$. To bound $\p(\cE^c \mid \cF_t)$, first observe that 
$$
\p(\cE^c) = \frac{1}{n} \sum_{u \in V_n} \p_v (\cE^c) \le e^{-b_2 r}
$$
by Lemma \ref{lemma:variance_upper_regime}. Markov's inequality then implies $\p(\cE^c \mid \cF_t) \le e^{-b_2 r/2}$ with probability at least $1 - e^{-b_2 r/2}$. Hence Lemmas \ref{lemma:variance_lower_regime} and \ref{lemma:variance_upper_regime} imply that with probability at least $1 - b_1 n^{-1/2} - e^{-b_2 r/2}$, 
\begin{align*}
\E [ \norm{\hat{v}(r) - \hat{v}(t)}_2^2 \mid \cF_t ] & \ge a_1 n^{2/d} - 4n^{2/d} e^{-b_2r/2} - 1 \\
& \ge \frac{a_1}{2} n^{2/d},
\end{align*}
where the last inequality holds for $n$ sufficiently large. Taking a union bound over $t \le H(\Delta_1 \log n)$ shows that with probability at least $1 - b_1 H(\Delta_1 \log n) n^{-1/2} - H(\Delta_1 \log n) e^{- b_2 r/2}$,
$$
\min\limits_{0 \le t \le G(\Delta_1 \log n)}\E [ \norm{\hat{v}(r) - \hat{v}(t)}_2^2 \mid \cF_t ]  \ge \frac{a_1}{2} n^{2/d},
$$
which in turn implies that $T_r \ge H(\Delta_1 \log n)$ with high probability. 
\end{proofof}

\section{Conclusion}
\label{sec:conclusion}

In this work we have formulated and developed a Bayesian approach to the problem of estimating the source of a cascade, given noisy time-series observations of the network. If vertices can be labelled by vectors in Euclidean space, we use optimal stopping theory to derive the Bayes-optimal stopping time. We then studied the performance of the optimal stopping time in lattices. Though the optimal estimator has a complex description, the stopping time $T_+$ which compares the posterior variance to a threshold is orderwise optimal. There are a number of future directions, including a rigorous study of the estimator \eqref{eq:optimal_estimator} when there is no vector labeling, and a performance analysis of other cascade dynamics and graphs.

\clearpage

\bibliographystyle{IEEEtran}
\bibliography{ciss_2020_final}

\clearpage

\appendix

\section{Useful properties of lattices}

\begin{lemma}
\label{lemma:lattice_estimates}
Let $G$ be the $d$-dimensional lattice, and suppose that vertices are labelled according to their lattice coordinates. Let $\partial \mathcal{N}_u(t)$ be the set of vertices with distance exactly $t$ from $u$. Then for any vertex $u \in V$, and any non-negative integer $t$,
$$
\frac{2^d}{(d-1)^{d-1}} (t - 1)^{d-1} \le | \partial \cN_u(t) | \le \frac{2^d e^{d-1}}{(d-1)^{d-1}} (t + d - 1)^{d - 1}.
$$
Furthermore, we can find positive constants $a,a',b,b'$ depending on $d$ such that $at^d \le |\cN(t)| \le bt^d$ and $a't^{d + 1} \le h(t) \le b' t^{d + 1}$. 
\end{lemma}

\begin{proof}
Let $\integers_+$ and $\integers_{\ge 0}$ denote the set of positive integers and non-negative integers, respectively. Define the sets 
\begin{align*}
S_+ & : = \left \{ x \in \integers_+^d : \sum\limits_{i = 1}^d x_i = t \right \} \\
S_{\ge 0} : &= \left \{ x \in \integers_{\ge 0}^d : \sum\limits_{i = 1}^d x_i = t \right \}.
\end{align*}
In words, $S_+$ is the collection of points in the positive quadrant of $\reals^d$ which are exactly distance $t$ from the origin. A similar interpretation holds for $S_{\ge 0}$. Since there are $2^d$ quadrants, the symmetry of lattices implies that
\begin{equation}
\label{eq:quadrant_inequality}
2^d | S_+ | \le | \partial \cN_u(t) | \le 2^d | S_{\ge 0} |.
\end{equation}
From standard counting arguments, $|S_{\ge 0} | = {t + d - 1 \choose d - 1}$ and $|S_+ | = {t - 1 \choose d - 1}$ for $t \ge d$. The bounds on $| \partial \cN(t)|$ follow from applying the bounds $\left(\frac{a}{b} \right)^b \le {a \choose b} \le \left( \frac{e\cdot a}{b} \right)^b$ to \eqref{eq:quadrant_inequality}. To prove the remaining two claims, we note that $|\partial \cN(t)| \asymp t^{d-1}$ and $|\cN(t)| = \sum_{s = 0}^t | \partial \cN(s)|$ so $|\cN(t)| \asymp t^d$. An analogous argument proves $h(t) \asymp t^{d + 1}$. 
\end{proof}

\begin{lemma}
\label{lemma:pnorm_sum}
For any $p \ge 1$, there are constants $c_{p,1}, c_{p,2}$ depending on $d$ and $p$ such that for $r \ge d$, 
\begin{equation*}
c_{p,1} r^{d + p}  \le \sum\limits_{u \in \cN_0(r)} \norm{u}_2^p \le c_{2,p} r^{d + p}.
\end{equation*}
\end{lemma}

\begin{proof}
Assume that $V_n = \cN_0(r)$ for some positive integer $r$. Since $d(u,v) = \norm{u - v}_1$ in lattices, we can write
\begin{equation}
\label{eq:pnorm_sum_decomp}
\sum\limits_{u \in \cN_0(r)} \norm{u}_1^p  = \sum\limits_{k = 1}^r k^p | \partial \cN_0(k) |
\end{equation}
By Lemma \ref{lemma:lattice_estimates}, we can find constants $c_1' = c_1'(d)$ and $c_2' = c_2'(d)$ such that $c_1' k^{d-1} \le | \partial \cN_0(k) | \le c_2' k^{d-1}$ for $k \ge d$. After plugging these bounds into \ref{eq:pnorm_sum_decomp}, we can find two more constants $c_1'' = c_1''(d)$ and $c_2''(d)$ such that 
$$
c_1'' r^{d + p} \le \sum\limits_{u \in \cN_0(r)} \norm{u}_1^p \le c_2'' r^{d + p}.
$$ 
Finally, due to the equivalence of norms in Euclidean space, we can find constants $c_1$ and $c_2$ such that the above equation holds when we replace $\norm{\cdot}_1$ with $\norm{\cdot}_2$. 
\end{proof}

\section{Characterizing the posterior variance}
In this section we prove Lemmas \ref{lemma:variance_lower_regime} and \ref{lemma:variance_upper_regime}. We begin by describing the distribution of the posterior probabilities, given by $\pi(t) : = \{\pi_u(t)\}_{u \in V_n}$. For any $u,v \in V$,
$$
\frac{ \pi_u(t)}{\pi_v(t)} = \frac{ \pi_u(t - 1)}{ \pi_v( t - 1)} \cdot \frac{ d \prob_u (y(t)) }{ d \prob_v(y(t))}.
$$
We can further decompose the second term on the right hand side above to obtain
\begin{multline*}
\frac{ d\prob_u( y(t)) }{ d\prob_v(y(t))} = \frac{ \prod\limits_{w \in \cN_u(t)} dQ_1(y_w(t)) \cdot \prod\limits_{w \notin \cN_u(t)} dQ_0(y_w(t)) }{ \prod\limits_{w \in \cN_v(t)} dQ_1(y_w(t)) \cdot \prod\limits_{w \notin \cN_v(t)} dQ_0(y_w(t)) } \\
 = \prod\limits_{w \in \cN_u(t) } \frac{ d Q_1}{d Q_0}(y_w(t)) \cdot \prod\limits_{w \in \cN_v(t) } \frac{ dQ_0}{dQ_1}(y_w(t)).
\end{multline*}
It follows that $\pi_u(t)$ has the form 
$$
\pi_u(t) = \frac{1}{Z(t)} \prod\limits_{s = 0}^t \prod\limits_{w \in \cN_u(s) } \frac{dQ_1}{dQ_0}(y_w(s)),
$$
where the normalizing constant is explicitly given by 
$$
Z(t) : = \sum\limits_{v \in V_n} \prod\limits_{s = 0}^t \prod\limits_{w \in \cN_v(s)} \frac{dQ_1}{dQ_0}(y_w(s)).
$$
It will be convenient to use the notation $\pi_u(t) : = X_u(t)/ Z(t)$, where $X_u(t)$ is explicitly given by 
$$
X_u(t) : = \prod\limits_{s = 0}^t \prod\limits_{w \in \cN_u(s)} \frac{dQ_1}{dQ_0}(y_w(s)).
$$
The following lemma establishes some basic properties of the collection $\{X_u(t)\}_{u \in V_n}$. 

\begin{lemma}
\label{lemma:xu_mean}
For any $u,v \in V_n$ and $t \ge 0$, $\E_v [ X_u(t) ] \ge 1$, with equality iff $d(u,v) > 2t$.
\end{lemma} 

\begin{proof}
Let $A \sim Q_0$ and $B \sim Q_1$. Then, by a change of measure and Jensen's inequality, 
\begin{align*}
\E \left [ \frac{dQ_1}{dQ_0}(B) \right] & = \E \left [ \left( \frac{dQ_1}{dQ_0}(A) \right)^2 \right] \ge  \E \left[ \frac{dQ_1}{dQ_0}(A) \right]^2 \ge 1.
\end{align*}
We remark that the inequality is strict if and only if $Q_0 \neq Q_1$, and that $\E \left [ \frac{dQ_1}{dQ_0}(B) \right] < \infty$ if and only if $Q_0$ and $Q_1$ are mutually absolutely continuous. It follows that, for any $v \in V_n$, 
\begin{align*}
\E_v [ X_u(t) ] & = \E_v \left [ \prod\limits_{s = 0}^t \prod\limits_{w \in \cN_u(s)} \frac{dQ_1}{dQ_0} (y_w(s)) \right] \\
& = \prod\limits_{s = 0}^t \prod\limits_{w \in \cN_u(s)} \E_v \left [ \frac{dQ_1}{dQ_0} (y_w(s)) \right] \ge 1.
\end{align*}
Equality is only possible if $y_w(s) \sim Q_0$ for all $w \in \cN_u(s)$ and $0 \le s \le t$, which in turn implies that $d(u,v) > 2t$. 
\end{proof}

\begin{lemma}
\label{lemma:xu_cov}
For any $u,v,w \in V_n$ and $t \ge 0$, there is a constant $\lambda = \lambda(Q_0,Q_1)$ such that $\mathrm{Cov}_v(X_u(t), X_w(t)) = 0$ if $d(u,w) > 2t$ and $\mathrm{Cov}_v(X_u(t), X_w(t)) \le \lambda^{h(t)}$ if $d(u,w) \le 2t$. 
\end{lemma}

\begin{proof}
It is clear from the structure of $X_u(t)$ and $X_w(t)$ that if $d(u,w) > 2t$, the two random variables are independent for any $v \in V_n$ and $\mathrm{Cov}_v(X_u(t), X_w(t)) = 0$ in this case. To handle the case where $d(u,w) \le 2t$, we first define 
\begin{align*}
\lambda_0 & : = \mathop{\E}_{A \sim Q_0} \left [ \left( \frac{dQ_1}{dQ_0} (A) \right)^2 \right] , \lambda_1  : = \mathop{\E}_{B \sim Q_1} \left [ \left( \frac{dQ_1}{dQ_0}(B) \right)^2 \right],
\end{align*}
and note in particular that $\lambda_0, \lambda_1 \ge 1$. We have the following bound on the covariance due to the Cauchy-Schwartz inequality. 
\begin{align*}
\mathrm{Cov}_v(X_u(t), X_w(t) ) & \le \E_v [ X_u(t) X_w(t) ] \\
& \le \E_v [ X_u(t)^2 ]^{1/2} \E_v [ X_w(t)^2 ]^{1/2}.
\end{align*}
To bound $\E_v [ X_u(t)^2 ]$, we can write
\begin{align*}
\E_v [ X_u(t)^2 ] & = \prod\limits_{s = 0}^t \prod\limits_{a \in \cN_u(s)} \E_v \left [ \left( \frac{dQ_1}{dQ_0}(y_a(s)) \right)^2 \right] \\
& =  \lambda_1^{ \sum_{s = 0}^t | \cN_u(s) \cap \cN_v(s) |} \lambda_0^{\sum_{s = 0}^t | \cN_u(s) \setminus \cN_v(s) |} \\
& \le (\lambda_0 \lambda_1)^{ \sum_{s = 0}^t | \cN_u(s) |} = (\lambda_0 \lambda_1)^{h(t)}
\end{align*}
It follows that $\mathrm{Cov}_v(X_u(t), X_w(t)) \le (\lambda_0 \lambda_1)^{h(t)}$, which proves the desired claim with $\lambda = \lambda_0\lambda_1$. 
\end{proof}

\begin{lemma}
\label{lemma:var_bounds}
For any $u,v,w \in V_n$ and $t \ge 0$,
\begin{equation}
\label{eq:cov_sum_2}
\sum\limits_{u,w \in V_n} \mathrm{Cov}_v (X_u(t), X_w(t)) \le n h(t) \lambda^{h(t)}.
\end{equation}
Furthermore, for all $p \ge 1$ there is a constant $C = C(d,p)$ such that
\begin{equation}
\label{eq:cov_sum_1}
\sum\limits_{u,w \in V_n} \norm{u}_2^p \norm{w}_2^p \mathrm{Cov}_v(X_u(t), X_w(t)) \le C n^{\frac{d + 2p}{d}} h(t) \lambda^{h(t)}.
\end{equation}
\end{lemma}

\begin{proof}
We focus on proving \eqref{eq:cov_sum_1}; the proof of \eqref{eq:cov_sum_2} is similar. First, note that if $d(u,w) = \norm{u-w}_1 \le 2t$, then 
$$
\norm{w}_2^p \stackrel{(a)}{\le} 2^{p-1} \norm{u}_2^p+ 2^{p-1} \norm{u - w}_2^p \stackrel{(b)}{\le} 2^{p-1} \norm{u}_2^p + 2^{2p - 1}t^p,
$$
where (a) is due to Jensen's inequality and (b) follows from $\norm{\cdot}_2 \le \norm{\cdot}_1$. Using the bound above, we can upper bound the summation in \eqref{eq:cov_sum_1} by 
\begin{multline*}
\sum\limits_{u,w \in V_n}  \left( 2^{p-1} \norm{u}_2^{2p} + 2^{2p - 1} t^p \norm{u}_2^p \right) \mathrm{Cov}_v(X_u(t), X_w(t)) \\
 \stackrel{(c)}{\le} \sum\limits_{u, w \in V_n : d(u,w) \le 2t} \left( 2^{p-1} \norm{u}_2^{2p} + 2^{2p - 1} t^p \norm{u}_2^p \right) \lambda^{g(t)} \\
 \stackrel{(d)}{\le} | \cN(2t)| \sum\limits_{u \in V_n} \left( 2^{p-1} \norm{u}_2^{2p} + 2^{2p - 1} t^p \norm{u}_2^p \right) \lambda^{(g(t)},
\end{multline*}
where (c) follows from Lemma \ref{lemma:xu_cov}, and (d) follows from noting that the terms in the summation do not depend on $w$. Next, an application of Lemma \ref{lemma:pnorm_sum} to bound $\sum_{u \in V_n} \norm{u}_2^{2p}$ and $\sum_{u \in V_n} \norm{u}_2^p$ allows us to bound the summation in \eqref{eq:cov_sum_1} by 
$$
| \cN(2t)| \left( C_1 n^{\frac{d + 2p}{d}} + C_2 t^p n^{ \frac{d + p}{d}} \right) \lambda^{g(t)},
$$
where $C_1$ and $C_2$ are constants depending on $p$ and $d$. Finally, \eqref{eq:cov_sum_1} follows by noting that $| \cN(2t)| \le g(t)$ and $n^{\frac{d + 2p}{d}}$ asymptotically dominates $t^p n^{ \frac{d + p}{d}}$ for $t \le n^{1/d}$. 
\end{proof}

The following lemma shows that $Z(t)$ is concentrated around $n$ when $t$ is not too large. 

\begin{lemma}
\label{lemma:Z}
Let $\epsilon \ge 1/\sqrt{n}$ and fix $v \in V_n$. Then there exists a constant $c_0 = c_0(Q_0,Q_1)$ such that for all $t \le G(c_0 \log n)$, 
$$
\p_v \left( | Z(t) - n | > \epsilon n \right) \le \frac{4}{\sqrt{n} \epsilon^2 }.
$$
\end{lemma}

\begin{proof}
We begin by computing the expectation of $Z(t)$ with respect to $\p_v$. 
\begin{align}
\E_v [ Z(t) ] & = \sum\limits_{u \in V_n} \E_v \left [\prod\limits_{s = 0}^t \prod\limits_{w \in \cN_v(s)} \frac{dQ_1}{dQ_0}(y_w(s)) \right] \nonumber \\
\label{eq:z_expansion1}
& = \sum\limits_{u \in V_n} \prod\limits_{s = 0}^t \prod\limits_{w \in \cN_u(s)} \E_v \left [ \frac{dQ_1}{dQ_0} (y_w(s)) \right].
\end{align}
Recall that conditioned on $v_0 = v$, $y_w(s) \sim Q_0$ if $w \notin \cN_v(s)$ else $y_w(s) \sim Q_1$. As a shorthand, denote $\alpha : = \E_{X \sim Q_1} \left [ \frac{dQ_1}{dQ_0} (X) \right]$; then from \eqref{eq:z_expansion1} we can write $\E_v [ Z(t)]$ as 
\begin{align*}
\E_v [ Z(t) ] & = \sum\limits_{u \in V_n} \prod\limits_{s = 0}^t \alpha^{ | \cN_v(s) \cap \cN_u(s) |} \\
& = \sum\limits_{u \in V_n} \alpha^{ \sum_{s = 0}^t | \cN_v(s) \cap \cN_u(s) |} \\
& = \sum\limits_{u : d(u,v) \le 2t} \alpha^{ \sum_{s = 0}^t | \cN_v(s) \cap \cN_u(s) |}  + | V_n \setminus \cN_v(2t) |.
\end{align*}
Since $\alpha \ge 1$, we have the bounds
$$
n \le \E_v[Z(t) ] \le n + | \cN_v(2t) | \alpha^{h(t)}.
$$
Note that if 
\begin{equation}
\label{eq:g_cond1}
h(t) \le \frac{\log n}{2(1 + \log \alpha)} \iff t \le H \left( \frac{\log n}{2( 1 + \log \alpha)} \right),
\end{equation}
then we have the following bound for any $\epsilon > 0$: 
\begin{equation}
\label{eq:epsilon_error}
n \le \E_v [ Z(t) ] \le (1 + \epsilon) n.
\end{equation}
Next, we turn to bounding the variance of $Z(t)$ under $\p_v$. By Lemma \ref{lemma:var_bounds}, 
$$
\mathrm{Var}_v(Z(t))  = \sum\limits_{u,w \in V_n} \mathrm{Cov}_v(X_u(t), X_w(t)) \le n h(t) \lambda^{h(t)}.
$$ 

Now we combine our first and second moment estimates of $Z(t)$ to obtain a concentration inequality. By Chebyshev's inequality, 
\begin{align}
\label{eq:chebyshev_bd}
\p_v \left( | Z(t) - \E_v [ Z(t) ] | > \epsilon n \right) & \le \frac{ \mathrm{Var}_v(Z(t)) }{n^2 \epsilon^2}  \le \frac{h(t) \lambda^{h(t)} }{n \epsilon^2}.
\end{align}
The right hand side above is at most $\frac{1}{\sqrt{n} \epsilon^2}$ if $t$ satisfies 
\begin{equation}
\label{eq:g_cond2}
h(t) \le \frac{\log n}{2(1 + \log \lambda )} \iff t \le H \left ( \frac{\log n}{1 + \log \lambda} \right).
\end{equation}
Set $\Delta_1^{-1} : = \max\{2( 1+ \log \alpha) , 2( 1 + \log \lambda ) \}$; then for $t \le H(\Delta_1 \log n)$, the statement of the lemma follows from \eqref{eq:epsilon_error} and \eqref{eq:chebyshev_bd}.
\end{proof}

We now turn to the proof of Lemma \ref{lemma:variance_lower_regime}. 

\begin{proof}[Proof of Lemma \ref{lemma:variance_lower_regime}]
Recall that conditioned on $\cF_s$, $v_0 \sim \pi(s)$ and $\E [ v_0 \mid \cF_s] = \hat{v}(s)$. Hence we have the following decomposition of the posterior variance: 
\begin{multline}
\label{eq:variance_decomposition}
\E [ \norm{v_0 - \hat{v}(s)}_2^2 \mid \cF_s ] = \sum\limits_{u \in V_n} \norm{u}_2^2 \pi_u(s) - \norm{\sum\limits_{u \in V_n} u \pi_u(s) }_2^2 \\
= \frac{1}{Z(s)} \sum\limits_{u \in V_n} \norm{u}_2^2 X_u(s) - \frac{1}{Z(s)^2} \norm{ \sum\limits_{u \in V_n} u X_u(s) }_2^2.
\end{multline}
To proceed, we derive a lower bound for the first term on the right hand side in \eqref{eq:variance_decomposition} and an upper bound for the second term on the right hand side in \eqref{eq:variance_decomposition}. Since $\E_v [ X_u(s) ] \ge 1$ by Lemma \ref{lemma:xu_mean}, we have the lower bound
$$
\E_v \left [ \sum\limits_{u \in V_n } \norm{u}_2^2  X_u(s) \right] \ge \sum\limits_{u \in V_n } \norm{u}_2^2 \ge c_1 n^{\frac{d+2}{d}},
$$
where the last inequality is a consequence of Lemma \ref{lemma:pnorm_sum} and $c_1$ is a function of $d$ only. By Lemma \ref{lemma:var_bounds}, we can bound the variance as  
$$
\mathrm{Var}_v \left( \sum\limits_{u \in V_n} \norm{u}_2^2 X_u(s) \right) \le Cn^{\frac{d + 4}{d}} h(t) \lambda^{h(t)}.
$$
Next, define the event
$$
\cE_1 : = \left \{ \sum\limits_{u \in V_n} \norm{u}_2^2 X_u(s) \ge \frac{c_1}{2} n^{\frac{d + 2}{d} } \right \}.
$$
Denoting $\bar{u}^2 : = \sum_{u \in V_n} \norm{u}_2^2 X_u(s)$ for brevity, Chebyshev's inequality implies
\begin{align*}
\p_v ( \cE_1^c ) & \le \p_v \left( | \bar{u}^2 - \E_v [ \bar{u}^2 ]| \ge \frac{c_1}{2} n^{\frac{d + 2}{d}} \right) \\
& \le \mathrm{Var}(\bar{u}^2) \cdot \frac{4}{c_1^2} n^{- \frac{2d + 4}{d} } \le \frac{4C}{c_1^2 n} h(t) \lambda^{h(t)}.
\end{align*}
Next, we establish a probabilistic upper bound for the term $\| \sum_{u \in V_n} u X_u(s) \|_2^2$. We have the decomposition
\begin{multline}
\label{eq:squared_expectation}
\E_v \left [ \norm{\sum\limits_{u \in V_n} u X_u(s) }_2^2 \right] = \norm{ \sum\limits_{u \in V_n} u \E_v [ X_u(s) ] }_2^2 \\
+ \sum\limits_{u,w \in V_n} ( u^\top w) \mathrm{Cov}_v(X_u(s), X_w(s)).
\end{multline}
We next establish an upper bound for the first term on the right hand side in \eqref{eq:squared_expectation}.
\begin{align*}
\norm{ \sum\limits_{v \in V_n} u \E_v [ X_u(s) ] }_2^2 & \stackrel{(a)}{=} \norm{ \sum\limits_{u \in V_n} u (\E_v [ X_u(s) ] - 1) }_2^2 \\
& \stackrel{(b)}{=} \norm{ \sum\limits_{u \in V_n : d(u,v) \le 2t} u (\E_v [ X_u(s) ] - 1) }_2^2 \\
& \stackrel{(c)}{\le} |\cN_v(2t) | \sum\limits_{u \in V_n : d(u,v) \le 2t} \norm{u}_2^2 \E_v [ X_u(s) ]^2 \\
& \stackrel{(d)}{\le} | \cN(2t)|^2 n^{2/d} \lambda^{2h(t)}.
\end{align*}
The equality (a) follows from $\sum_{u \in V_n} u = 0$, (b) is due to $\E_v [ X_u(s) ] = 1$ for $d(u,v) > 2t$, (c) is due to Jensen's inequality and (d) follows from $\E_v [ X_u(s) ] \le \E_v [ X_v(s) ] = \lambda^{g(t)}$ and $\norm{u}_2 \le \norm{u}_1 \le n^{1/d}$.

We next derive an upper bound for the second term on the right hand side in \eqref{eq:squared_expectation} using Lemma \ref{lemma:var_bounds}:
\begin{multline*}
\sum\limits_{u,w \in V_n} u^\top w \mathrm{Cov}_v(X_u(s), X_w(s)) \\
 \le \hspace{-0.25cm} \sum\limits_{u,w \in V_n} \norm{u}_2 \norm{w}_2 \mathrm{Cov}_v(X_u(s), X_w(s)) \le C_2 n^{\frac{d + 2}{d}} h(t) \lambda^{h(t)}.
\end{multline*}
Above, $C_2$ is the constant that comes from the lemma with $p = 1$. From \eqref{eq:g_cond2}, if $t \le H\left( \frac{\log n}{1 + \log \lambda} \right)$ then $h(t) \lambda^{h(t)} \le \sqrt{n}$ and the second term on the right hand side of \eqref{eq:squared_expectation} dominates for large $n$. We can therefore bound, for $n$ sufficiently large, 
$$
\E_v \left [ \norm{ \sum\limits_{u \in V_n} u X_u(s) }_2^2 \right] \le 2C_2 n^{\frac{d + 2}{d}} g(t) \lambda^{g(t)} \le 2C_2 n^{\frac{d + 2}{d} + \frac{1}{2}}.
$$
Define the events $\cE_2 := \{ \| \sum_{u \in V_n} u X_u(s) \|_2^2 \le n^{\frac{d+2}{d} + \frac{5}{6}} \}$ and $\cE_3 : = \{ n/2 \le Z(t) \le 3n/2 \}$. By Markov's inequality, $\p_v(\cE_2^c) \le 2C_2 n^{-1/3}$ and $\p_v(\cE_3^c) \le 16n^{-1/2}$ by Lemma \ref{lemma:Z}. On the event $\cE_1 \cap \cE_2 \cap \cE_3$, we can lower bound \ref{eq:variance_decomposition} by $\frac{c_1}{3} n^{2/d} - 4n^{2/d - 1/6} \ge \frac{c_1}{6} n^{2/d}$, where the latter inequality holds for $n$ sufficiently large. We conclude by noting that $\p(\cE_1^c \cup \cE_2^c \cup \cE_3^c) \le C_3 n^{-1/2}$ for some constant $C_3$ depending only on $d$, which follows from a union bound. 
\end{proof}

Finally, we turn to the proof of Lemma \ref{lemma:variance_upper_regime}.

\begin{proof}[Proof of Lemma \ref{lemma:variance_upper_regime}]
For any vertex $v \in V_n$, we have 
\begin{align*}
\E \left [ \norm{v_0 - \hat{v}(t)}_2^2 \mid \cF_t \right] & \stackrel{(a)}{\le} \sum\limits_{u \in V_n} \norm{u - v}_2^2 \pi_u(t) \\
& \stackrel{(b)}{\le} \sum\limits_{u \in V_n} \norm{u - v}_2^2 \frac{ X_u(t)}{X_v(t)},
\end{align*}
where (a) follows from properties of the variance and (b) follows since $\pi_u(t) = X_u(t) / Z(t)$ and $Z(t) \ge X_v(t)$. For any two vertices $u,v$, define the quantity $h_{vu}(t) : = \sum_{s = 0}^t | \cN_v(t) \setminus \cN_u(t) |$ as well as the event 
$$
\cE_{vu} : = \left \{ \log \frac{X_v(t)}{X_u(t)} \ge h_{vu}(t) \left( \frac{ D(Q_0 \| Q_1) + D(Q_1 \| Q_0)}{2} \right) \right\},
$$
where $D(Q_0 \| Q_1)$ is the Kullback-Leibler divergence between $Q_0$ and $Q_1$. As a shorthand, we denote $D : = \frac{1}{2} ( D(Q_0 \| Q_1) + D(Q_1 \| Q_0) )$. A Chernoff bound argument \cite[Theorem 2]{sridhar_poor_2019} implies that $\p_v(\cE_{vu}^c) \le e^{- c h_{vu}(t)}$, where $c$ is a constant depending only on $Q_0$ and $Q_1$. On the event $\cE_v : = \bigcup_{u \in V_n \setminus \{v\}} \cE_{vu}$, we can bound
\begin{multline}
\label{eq:squared_diff_sum}
\sum\limits_{u \in V_n} \norm{u - v}_2^2 \frac{X_u(t)}{X_v(t)}  \stackrel{(a)}{\le} \sum\limits_{u \in V_n} \norm{u - v}_2^2 e^{- D h_{vu}(t) } \\
\stackrel{(b)}{\le} \sum\limits_{u \in V_n} d(u,v)^2 e^{- D h_{vu}(t) }  \stackrel{(c)}{=} \sum\limits_{k = 1}^{n^{1/d}} k^2 | \cN_v(k)| e^{-D h_{k}(t) },
\end{multline}
where (a) holds on $\cE_v$, (b) follows from the inequality $\norm{\cdot}_2 \le \norm{\cdot}_1$, and (c) follows from grouping together the terms that are equidistant from $v$, where we define $h_k(t) := h_{vu}(t)$ for any $u$ such that $d(u,v) = k$.\footnote{Such a definition is valid due to vertex transitivity of lattices.} To simplify the final summation in \eqref{eq:squared_diff_sum}, we split it into terms corresponding to $k \le 2t$ and those corresponding to $k > 2t$. If $k \le 2t$ then we use the simple bounds $h_{k}(t) \ge t$ and $|\cN(k)| \le h(2t)$ to obtain $\sum_{k = 1}^{2t} k^2 | \cN_v(k) | e^{-D h_k(t) } \le 8t^3 h(2t) e^{-D t} \le e^{-Dt/2}$ for $t$ sufficiently large since $h$ grows polynomially. Next, noting that $h_k(t) = h(t)$ for $k > 2t$, we can bound the remaining terms as 
\begin{align}
\sum\limits_{k = 2t + 1}^{n^{1/d}} k^2 | \cN_v(k) | e^{- D h(t) } & \stackrel{(d)}{\le} n^{2/d} e^{-D h(t)} \sum\limits_{k = 2t + 1}^{n^{1/d}} | \cN_v(k) | \nonumber \\
\label{eq:squared_diff_sum_2}
& \stackrel{(e)}{\le} n^{1 + 2/d} e^{ - D h(t) },
\end{align}
where (d) follows from $k \le n^{1/d}$ and (e) follows since the sets $\{ \cN_v(k) \}_{k =0}^{n^{1/d}}$ partition $V_n$. If we assume that $t \ge H(6D^{-1} \log n)$, then the final summation in \eqref{eq:squared_diff_sum_2} is at most $e^{-Dt/2}$ as well. It remains to bound $\p_v(\cE_v^c)$. We can write
\begin{multline}
\label{eq:final_prob_bound}
\p_v(\cE_v^c) \stackrel{(f)}{\le} \sum\limits_{u \in V_n \setminus \{v \}} \p_v (\cE_{vu}^c)  \\
\stackrel{(g)}{\le} \sum\limits_{u \in V_n \setminus \{v \}} e^{-c h_{vu}(t) } \stackrel{(h)}{=} \sum\limits_{k = 1}^{n^{1/d}} | \cN_v(k)| e^{-c h_k(t) },
\end{multline}
where (f) is due to a union bound, (g) follows from a Chernoff bound argument \cite[Theorem 2]{sridhar_poor_2019}) and (h) follows from identical reasoning as (c) in \eqref{eq:squared_diff_sum}. We can bound the final summation in \eqref{eq:final_prob_bound} using identical reasoning as before, i.e., splitting the summation into terms corresponding to $k \le 2t$ and $k > 2t$. We arrive at a final bound of $h(2t) e^{-c t} + n e^{- c h(t) }$, which is at most $e^{-ct/2}$ for $t \ge H(2c^{-1}\log n)$. The desired statement follows from setting $\Delta_2 := \max\{6D^{-1},2c^{-1}\}$. 
\end{proof}

\end{document}